\documentclass[a4paper,oneside,10pt,reqno]{amsart}
\usepackage{color,enumerate,tikz,subfigure}
\usepackage{amssymb}
\usepackage{epic,cite}
\usepackage{stmaryrd}
\usetikzlibrary{fit,calc}
\usepackage[multiple]{footmisc}

\usepackage[utf8]{inputenc}
\usepackage[T1]{fontenc}

\newcommand{\N}{\mathbb{N}}
\newcommand{\R}{\mathbb{R}}

\newcommand{\incomp}{\mathbin{||}}

\newcommand\restr[2]{{
  \left.\kern-\nulldelimiterspace 
  #1 
  \right|_{#2} 
  }}

\newtheorem{thm}{Theorem}[section]
\newtheorem{cor}[thm]{Corollary}
\newtheorem{lem}[thm]{Lemma}
\newtheorem{prop}[thm]{Proposition}

\theoremstyle{definition}
\newtheorem{defn}[thm]{Definition}
\newtheorem{ex}[thm]{Example}
\newtheorem{fact}[thm]{Fact}
\theoremstyle{remark}
\newtheorem{rem}[thm]{Remark}
\numberwithin{equation}{section}
\newtheoremstyle{case}{}{}{}{}{\itshape}{:}{ }{}
\theoremstyle{case}

\definecolor{gris10}{gray}{0.9}
\definecolor{gris}{gray}{0.5}

\def\tkzscl{0.8}

\title[Asociative, idempotent, symmetric, order-preserving op. on chains]{Associative, idempotent, symmetric, and order-preserving operations on chains}

\author[J. Devillet]{Jimmy Devillet}
\author[B. Teheux]{Bruno Teheux}
\date{\today}
\begin{document}

\maketitle

\begin{abstract}
We characterize the associative, idempotent, symmetric, and order-preserving operations on (finite) chains in terms of properties of (the Hasse diagram of) their associated semilattice order. In particular, we prove that the number of associative, idempotent, symmetric, and order-preserving operations on an $n$-element chain is the $n^{\text{th}}$ Catalan number.
\end{abstract}
\section{Introduction}

%
The associativity equation for binary operations is ubiquitous in mathematics, as many algebraic structures are defined with associative operations (semigroups, groups, rings, Lie groups, \emph{etc}).  Associative operations also appear  in the algebraic treatment of classical and non-classical logics \cite{HandbookFuzz, Humberstone2011}.   They have also been studied by several authors in the theory of functional equations (see, e.g.,\cite{Acz2006,CzoDre84,Alsina2006} and the references therein). Moreover, associativity has been considered in conjunction with other properties, such as idempotency  (see, e.g., \cite{Kimura1958,McLean1954}) and quasitriviality (see, e.g., \cite{Ackerman,Langer1980,Couceiro2017}).
The class of associative, idempotent, and symmetric operations is of particular interest since it is in one-to-one correspondence with the class of partial orders of semilattices  (see, e.g., \cite{Gratzer2003}).


Let $(X,\leq)$ be an arbitrary totally ordered set. In this paper, we investigate the class of associative, idempotent, symmetric, and $\leq$-preserving  operations $F\colon X^2 \to X$. The subclass of those operations that are quasitrivial, that is, where $F(x,y)\in\{x,y\}$ for all $x,y\in X$, was investigated in \cite{Couceiro2017}. In particular, this subclass was characterized on both arbitrary sets and finite sets (see \cite[Theorem 3.7]{Couceiro2017}) in terms of their associated total order. This characterization uses the single-peakedness property that was introduced 70 years ago in Social Choice theory (see, e.g., \cite{Bla48,Bla87}). The aim of this paper is to obtain a more general characterization by relaxing quasitriviality into idempotency. In the case where $X$ is finite, we particularize this characterization in terms of properties of the Hasse diagram of the corresponding semilattice. We also obtain several enumeration of classes of semilattices. In this respect, one of our main results is a new occurrence of the Catalan numbers as the numbers of associative, idempotent, symmetric, and nondecreasing operations on finite chains.

The outline of this paper is as follows. In Section \ref{sec:prel} we set up the definitions and terminology used in this paper. In Section \ref{sec:car}, we characterize (Theorem \ref{thm:mainic}) by means of a generalization of the single-peakedness property  the class of associative, idempotent, and symmetric operations $F\colon X^2 \to X$ that are nondecreasing with respect to $\leq$.
In Section \ref{sec:finite}, we characterize (Theorem \ref{thm:main}) the same class, denoted by $C_n$,  when $X$ is a finite set of cardinality $n\geq 1$. This characterization is stated in terms of properties of the Hasse diagram of the semilattice $(X,F)$ . This result leads to an associativity test for idempotent, symmetric, and nondecreasing operations on finite chains. Moreover, we prove that the cardinality of $C_n$ is the $n^{\text{th}}$ Catalan number (Proposition \ref{prop:cat}), providing yet another construction of the sequence of Catalan numbers. Given a finite semilattice operation $\curlyvee$, we also consider the problem of enumerating the total orders for which  $\curlyvee$ is nondecreasing (Corollary \ref{cor:rev}). In section \ref{sec:kary}, we combine our characterizations with recent results about reducibility of $k$-ary associative   operations \cite{Kiss}.

\section{Preliminaries}\label{sec:prel}
In this section, we set up the  notation and terminology. Recall that an operation $F\colon X^2 \to X$ is said to be

\begin{itemize}
\item \emph{associative} if $F(x, F(y,z))=F(F(x,y),z)$ for all $x,y,z\in X$,
\item \emph{symmetric} if $F(x,y)=F(y,x)$ for all $x,\in X$,
\item \emph{idempotent}  if  $F(x,x)=x$ for all $x\in X$.
\end{itemize}

We refer to \cite{Davey2002} for an introduction to order theory.
A \emph{partially ordered set} is an ordered pair $(X,\preceq)$, where $X$ is a set and $\preceq$ is a \emph{partial order} on $X$, that is, a reflexive, antisymmetric, and transitive relation on $X$. We denote by $\prec$ the \emph{asymmetric part} of $\preceq$, that is, we have $x\prec y$ if and only if $x\preceq y$ and $y\not\preceq x$. We write $x\incomp y$ and  if $x$ and $y$ are \emph{incomparable}, that is, if $x\not\preceq y$ and $y\not\preceq x$. For any  $Y\subseteq X$, we denote by $\preceq_Y$ the \emph{restriction of $\preceq$ to $Y$}.
For simplicity, we often write $(Y, \preceq)$ for $(Y, \preceq_Y)$. A partial order $\leq$ on $X$ is said to be \emph{total}, and $(X, \leq)$ is called a \emph{totally ordered set} or a \emph{chain} if for every $x,y \in X$, we have $x\leq y$ or $y \leq x$. In this paper we use the notation $\leq$ for total orders.

An element $z$ of a partially ordered set $(X, \preceq)$ is an \emph{upper bound} of  $Y\subseteq X$  if $y\preceq z$ for every $y\in Y$. An upper bound $z$ of $Y$ is a \emph{supremum} of $Y$ if $z\preceq z'$ for every upper bound $z'$ of $Y$. \emph{Lower bounds} and \emph{infimum} are defined dually. Partial orders $\preceq$ on $X$ for which every pair  $\{x,y\}\subseteq X$ has a \emph{supremum} $x\curlyvee y$ are called \emph{join-semilattice orders}, and in this case $(X,\preceq)$ is a called a \emph{join-semilattice}. If $\preceq$ is a join-semilattice order, then it is known that the \emph{join operation} $\curlyvee \colon X^2\to X$ defined by $\curlyvee (x,y)= x\curlyvee y$ is associative, symmetric, and idempotent, and the 2-tuple $(X,\curlyvee)$ is called the \emph{semilattice associated with $\preceq$}.  We denote by $\vee$ the join operation of a total order $\leq$. It is easily seen that such an operation $\vee$ is quasitrivial. Groupoids  $(X,F)$ where $F$ is associative, symmetric and idempotent are called \emph{semilattices}, and $F$ is a \emph{semilattice operation}. It is well known that  every semilattice $(X,F)$ is the join-semilattice associated with the partial order $\preceq_F$ defined as
\begin{equation}\label{eqn:ord}
x\preceq_F y\quad \text{if} \quad F(x,y)=y.
\end{equation}
That is, the join operation of $\preceq_F$ is $F$.  We say that $\preceq_F$ is the \emph{(join-semilattice) order associated with $F$}, and we denote $\preceq_F$ by $\preceq$ if no confusion is possible. The semilattice operation $F$ is  quasitrivial if and only if $\preceq_F$ is a total order. The mappings $(X,\preceq) \mapsto (X,\curlyvee)$ and $(X,F) \mapsto (X,\preceq_F)$ are inverse to each other, and define a one-to-one correspondence between  join-semilattices and semilattices. By this correspondence,  we use either  $(X,\preceq)$ or $(X,\curlyvee)$ to denote a join-semilattice, and every semilattice order will be a join-semilattice order. Thanks to this convention, we write \emph{semilattice} for \emph{join-semilattice}.

A nonempty subset $I$  of a semilattice $(X, \curlyvee)$ is an \emph{ideal} if it is a lower set closed under $\curlyvee$, that is, if $x\in X$ and $y,z\in I$ are such that $x\leq y$, then $x\in I$ and $y\curlyvee z\in I$. A nonempty subset $H$ of a partially ordered set $(X,\preceq)$ is a \emph{filter} if it is a dually directed upper set, that is, if $x\in X$ and $y,z\in F$ are such that $y\preceq x$, then $x\in F$ and there is $t\in F$ such that $t\preceq y$ and $t\preceq z$. For every $x\in X$, the sets $(x]_\preceq=\{y\in X \mid y\preceq x\}$ and $[x)_\preceq=\{y\in X \mid x\preceq y\}$ are the \emph{ideal} and the \emph{filter} generated by $x$, respectively.
 An ideal $I$ is \emph{principal} if there is $x\in X$ such that $I=(x]$. \emph{Principal filters} are defined dually. In particular, in a finite semilattice, all filters and ideals are principal. In a totally ordered set $(X, \leq)$, we set  $[a,b]=\{x\in X  \mid a\leq x \leq b\}$ for every $a\leq b$ in $X$. A subset $C$ of $X$ is said to be \emph{convex} if it contains $[a,b]$ for any $a,b\in C$ with $a<b$.

Finite semilattices whose Hasse diagram is a binary tree are of special interest. We call them \emph{binary semilattices}.

\section{Order-preserving semilattice operations}\label{sec:car}

Let $\leq$ be a total order on a set $X$. An operation $F\colon X^2 \to X$ is said to be   \emph{$\leq$-preserving} if $F(x,y)\leq F(z,t)$ for every $x\leq z$ and $y\leq t$ in $X$. The main result of  this section characterizes the class of $\leq$-preserving semilattice operations in terms of their associated partial order (see Theorem \ref{thm:mainic}). The characterizing property, called \emph{nondecreasingness}, is introduced in Definition \ref{def:ic}. Single-peaked total orders (in the sense of \cite{Bla48,Bla87}, see also \cite{Couceiro2018,Couceiro2017}) are instances of nondecreasing partial orders. It follows that Theorem \ref{thm:mainic}  and \ref{thm:main} generalize \cite[Proposition 3.9]{Devillet2018} and  \cite[Proposition 5.1]{Devillet2018} by relaxing quasitriviality into idempotency.

\begin{defn}\label{def:ic}
Let $(X,\leq)$ be a chain. We say that a semilattice order $\preceq$ on $X$ has
the \emph{convex-ideal property} (\emph{CI-property} for short) \emph{for~$\leq$} if for every $a,b,c \in X$,
\begin{equation}\label{eqn:ic}
a\leq b \leq c \quad \implies \quad  b\preceq a  \curlyvee c.
\end{equation}
We say that $\preceq$ is \emph{internal for~$\leq$} if for every $a,b,c\in X$,
\begin{equation}\label{eqn:it}
a< b < c \quad \implies \quad (a\not\neq b \curlyvee c \quad \text{ and } \quad c\neq a\curlyvee b).
\end{equation}
We say that $\preceq$ is  \emph{nondecreasing  for~$\leq$} and  that the semilattice $(X, \curlyvee)$ is \emph{nondecreasing for $\leq$}  if $\preceq$ has the CI-property and is internal for $\leq$.
\end{defn}

Note that if $\preceq$ and $\leq$ are total orders on $X$, then $\preceq$ is nondecreasing for $\leq$ if and only if it is single-peaked for $\leq$ in the sense of  \cite{Bla48,Bla87}, that is, if and only if condition \eqref{eqn:ic} is satisfied. The terminology introduced in Definition \ref{def:ic} is justified in Lemmas \ref{lem:ic}  and \ref{lem:int}, and Theorem   \ref{thm:mainic}. Note that conditions \eqref{eqn:ic} and \eqref{eqn:it} are self-dual w.r.t.~the total order $\leq$, \emph{i.e.}, if $\leq^\alpha$ is the dual order of $\leq$ (that is, $x \leq^\alpha y$ if and only if $y \leq x$),  then $\leq$ satisfies \eqref{eqn:ic} and \eqref{eqn:it} if and only if $\leq^\alpha$ satisfies \eqref{eqn:ic} and \eqref{eqn:it}, respectively.

\begin{fact}\label{fact:res}
Let $(X,\leq)$ be a totally ordered set, let $\preceq$ be a semilattice order on $X$, and let $P \subseteq X$.  If the restriction $\preceq_P$ of $\preceq$ to $P$ is a total order, then $\preceq_P$ is nondecreasing for~$\leq_P$ if and only if  it is single-peaked for $\leq_P$ in the sense of  \cite{Bla48,Bla87}.
\end{fact}

It follows from Fact \ref{fact:res} that those total orders that are nondecreasing for a given total order $\leq$ are exactly  the single-peaked ones.
\begin{rem}
Condition \eqref{eqn:ic} is clearly equivalent to
\[
a< b < c \quad \implies \quad  b\preceq a  \curlyvee c,
\]
but the partial order $\preceq$ cannot be replaced by its irreflexive part in \eqref{eqn:ic}. Indeed, if $\leq$ is the natural order on  $X_3=\{1,2,3\}$ and if $\preceq$ is the partial semilattice order defined as $1\preceq 2$, $3\preceq 2$, and $1\incomp3$, then $\preceq$ satisfies   \eqref{eqn:ic} but $2\not\prec 1\curlyvee 3$.

\end{rem}

The following lemma is a generalization of \cite[Proposition 3.10]{Devillet2018} for nondecreasing semilattice orders.

\begin{lem}\label{lem:ic}
Let $(X,\leq)$ be a totally ordered set and $\preceq$ be a  semilattice order on $X$. The following conditions are equivalent.
\begin{enumerate}[(i)]
\item\label{it:bfr01} The semilattice order $\preceq$ has the CI-property for $\leq$.
\item\label{it:bfr02} Evey ideal of $(X, \preceq)$ is a convex subset of $(X, \leq)$.
\item\label{it:bfr03} Every principal ideal of $(X, \preceq)$ is a convex subset of $(X, \leq)$.
\item\label{it:bfr04} If $x'\preceq x$ or $x'\incomp x$ then $x$ is an upper bound or a lower bound of $(x']_\preceq$ in $(X,\leq)$.
\end{enumerate}
\end{lem}
\begin{proof}
(\ref{it:bfr01}) $\implies$ (\ref{it:bfr02}) Let $I$ be an ideal of $(X, \preceq)$ and let $a, c\in I$. For every $b\in X$ such that $a\leq b \leq c$ we have $b \preceq a\curlyvee c$ by the CI-property. If follows that $b\in I$ since $I$ is an ideal of $(X,\preceq)$ that contains $a$ and $c$.

(\ref{it:bfr02}) $\implies$ (\ref{it:bfr03}) Obvious.

(\ref{it:bfr03}) $\implies$ (\ref{it:bfr01}) Let $a\leq b \leq c$ in $X$. By (\ref{it:bfr03}), the ideal $(a\curlyvee c]_\preceq$ is convex in $(X, \leq)$. Since it contains $a$ and $c$, it also contains $b$. It follows that $b\preceq a \curlyvee c$.

(\ref{it:bfr03}) $\iff$ (\ref{it:bfr04}) Obvious.
\end{proof}

Now we give equivalent formulations of the internality property \eqref{eqn:it} for semilattice orders. Recall that a binary operation $F\colon X^2 \to X$ defined on a totally ordered set $(X,\leq)$ is said to be \emph{internal} \cite{Grabisch2009} if $x\leq F(x,y) \leq y$ for every $x, y\in X$.

\begin{lem}\label{lem:int}
Let $(X,\leq)$ be a chain and let $\preceq$ be a join-semilattice order on $X$. The following conditions are equivalent.
\begin{enumerate}[(i)]
\item\label{it:mlo01} The partial order $\preceq$ is internal for ~$\leq$.
\item\label{it:mlo02} The join operation of $\preceq$ is internal.
\item\label{it:mlo03} There are no $a,b,c\in X$ such that $a<b<c$  and $a\curlyvee b=b\curlyvee c\in \{a,c\}$.
\end{enumerate}
Moreover, if one of these conditions is satisfied, then there are no pairwise $\preceq$-incomparable elements $a,b,c$ of $X$ such that $a\curlyvee b=a \curlyvee c=b\curlyvee c$.
\end{lem}
\begin{proof}
(\ref{it:mlo01}) $\implies$ (\ref{it:mlo02}) For any $a<b$ in $X$, we cannot have $a\curlyvee b <a<b$ or $a<b<a\curlyvee b$, since this would contradict  internality of $\preceq$ for $\leq$. It follows that $a\curlyvee b \in [a,b]$.

(\ref{it:mlo02}) $\implies$ (\ref{it:mlo03}) Let $a<b<c$ in $X$. If $a\curlyvee b=b\curlyvee c=a$, then $b\curlyvee c\not\in [b,c]$.  If $a\curlyvee b=b\curlyvee c=c$, then $a\curlyvee b\not\in [a,b]$.

(\ref{it:mlo03}) $\implies$ (\ref{it:mlo01}) We show the contrapositive. Assume that there are $a<b<c$ such that $a=b\curlyvee c$. Then $a\curlyvee b=b\curlyvee c=a$. Similarly, if $c=a\curlyvee b$ then $b\curlyvee c= a\curlyvee b=c$.

Now, assume that one of the conditions of the statement is satisfied, and that $a<b<c$ are pairwise $\preceq$-incomparable elements of $X$. If $a\curlyvee b$ and $b\curlyvee c$ are equal to a common element $d$,  it follows from  (\ref{it:mlo02}) that $d\in \left]a,b\right[\cap\left]b,c\right[=\varnothing$, a contradiction.
\end{proof}

\begin{cor}
Let $(X,\leq)$  be a chain and  $F\colon X^2 \to X$ be an  operation. Then $F$ is associative, symmetric and internal if and only if $F$ is the join operation of a semilattice order that is internal for $\leq$.
\end{cor}

\begin{rem}
If $\preceq$ is a semilattice order that is internal for a total order $\leq$, there might be $\preceq$-incomparable elements $a,b,c$ such that $a\curlyvee b=b\curlyvee c$. Consider for instance $X=\{a,b,c,d,e\}$ with $a<e<c<d<b$ and the semilattice order $\preceq$ defined by $a\incomp b$, $a\incomp c$, $c\incomp b$, $a\curlyvee c=e$ and $e\curlyvee b=d$. Then $\preceq$ is internal for $\leq$ and $a\curlyvee b=b\curlyvee c$.
\end{rem}

\begin{rem}\label{ex:base}
The join operation $\curlyvee$ of a semilattice order $\preceq$ that has the CI-property for a total order $\leq$ need not be $\leq$-preserving. For instance, if $\leq$ is the natural order on $X_3=\{1,2,3\}$ and $\preceq$ is defined on $X$ as $2\preceq 1$, $3\preceq 1$ and $2\incomp 3$, then $\preceq$ has the CI-property for $\leq$ but $2\curlyvee 2=2>1=2\curlyvee 3$. This example also shows that the CI-property for $\leq$ does not imply internality for $\leq$.

Conversely,  the join operation $\curlyvee$ of a semilattice order $\preceq$ that is internal for a total order $\leq$ needs not to be $\leq$-preserving. For instance, if $\preceq$ is the total order $1\preceq 3 \preceq 2$, then $\preceq$ is internal for $\leq$ but $2=2\curlyvee 1 > 1\curlyvee 1=1$. This example also shows that internality for $\leq$ does not imply CI-property for $\leq$.
\end{rem}

\begin{thm}\label{thm:mainic}
Let $(X, \leq)$ be a totally ordered set and $ \curlyvee \colon X^2 \to X$ be a semilattice operation. The following conditions are equivalent.
\begin{enumerate}[(i)]
\item\label{it:opi01} $\curlyvee$ is $\leq$-preserving.
\item\label{it:opi02} The order $\preceq$ associated with $\curlyvee$ is nondecreasing for $\leq$.
\end{enumerate}
\end{thm}
\begin{proof}
(\ref{it:opi01}) $\implies$ (\ref{it:opi02}) First, we prove that $\preceq$ has the CI-property for $\leq$.  Let $a<b<c$ in $X$. Since $\curlyvee$ is $\leq$-preserving, we obtain
\[
a\curlyvee c=a \curlyvee (a\curlyvee c)\leq b \curlyvee (a\curlyvee c) \leq c \curlyvee (a\curlyvee c)=a\curlyvee c.
\]
It follows that $b \curlyvee (a\curlyvee c)=a\curlyvee c$, which proves that $b\preceq (a \curlyvee c)$.

Then, we prove that $\preceq$ is internal for $\leq$. By Lemma \ref{lem:int}, it suffices to note that for every $a<b$ in $X$, we have
\[
a=a\curlyvee a  \leq a \curlyvee b  \leq b \curlyvee b =b,
\]
since $\curlyvee$ is $\leq$-preserving.

(\ref{it:opi02}) $\implies$ (\ref{it:opi01}) For the sake of contradiction, assume that there are $a,b,c\in X$ such that $b<c$ and $a \curlyvee c<a \curlyvee b$.
Assume first that $b \curlyvee c<a \curlyvee c<a \curlyvee b$.  It follows by internality of $\preceq$ for $\leq$ that $a\not\preceq b \curlyvee c$ and $c\not\preceq a \curlyvee b$ . Since $\preceq$ has the CI-property for $\leq$, we obtain that $\neg (b<a<c)$ and $\neg (b<c<a)$, so $a<b<c$. From $b<c$ and Lemma  \ref{lem:int}, we deduce that $a<a\curlyvee b\leq b<c$. It follows that $b \curlyvee c<b<c$, which contradicts internality of $\preceq$ for $\leq$ by Lemma \ref{lem:int}.

The case $a  \curlyvee c<a  \curlyvee b<b  \curlyvee c$ follows by duality, and the case $a  \curlyvee c<b  \curlyvee c<a  \curlyvee b$ is obtained similarly.
\end{proof}

The following result is a direct consequence of Theorem \ref{thm:mainic}.
\begin{thm} \label{cor:main}
Let $(X, \leq)$ be a totally ordered set and let $F\colon X^2 \to X$ be an operation. The following conditions are equivalent
\begin{enumerate}[(i)]
\item $F$  is a $\leq$-preserving semilattice operation.
\item $F$ is the join operation of a  semilattice order $\preceq$ on $X$ that is nondecreasing for $\leq$.
\end{enumerate}
\end{thm}

Theorem \ref{thm:mainic} will be  particularized for finite chains $(X,\leq)$ in Theorem \ref{thm:main} and illustrated in Example \ref{ex:main}

We close this section by showing that if $\preceq$ is a semilattice order that is nondecreasing for a total order $\leq$, then every filter of $(X,\preceq)$ is totally ordered.

\begin{defn}
A partial order $\preceq$  on $X$ is sad to have the \emph{linear filter property} if every of its filter is totally ordered.
\end{defn}

\begin{lem}\label{lem:linf}
A partial order on $X$ has the linear filter property if and only if  no pair $\{a,b\}$ of incomparable elements of $X$  has a lower bound.
\end{lem}
\begin{proof}
(Necessity)
If $\preceq$ is a partial order on $X$ having the linear filter property and there is a pair $\{a,b\}$ of incomparable elements of $X$  that has a lower bound $c$, then $[c)$ is a filter that is not totally ordered.

(Sufficiency) Obvious.
\end{proof}

%
\begin{prop}\label{lem:icintlin}
Let $(X,\leq)$ be a totally ordered set and $\preceq$ be a semilattice order on $X$. If $\preceq$ is nondecreasing for $\leq$, then it has the linear filter property.
\end{prop}

\begin{proof} We prove the contrapositive. Assume that $\preceq$ does not have the linear filter property. By Lemma \ref{lem:linf},  there are $\preceq$-incomparable elements $a, b$ in $X$ that have a lower bound $c$. Let us set $d=a\curlyvee b$ and  assume  that $a<b$. If $d<a$ or $b<d$, then $\preceq$ is not internal for $\leq$. Assume that $a<d<b$. If $c<a<d<b$ then $\preceq$ does not have the CI-property for $\leq$ since $d\not\preceq c\curlyvee b=b$. The cases $a<d<b<c$, $a<c<d<b$, and $a<d<c<b$ can be deal with similarly.
\end{proof}

\begin{rem}
The converse of Proposition \ref{lem:icintlin} does not hold. On the one hand, Example \ref{ex:base} shows an instance of a semilattice order that has the linear filter property but that is not internal for~a given total order $\leq$. On the other hand, if $\leq$ is the natural order on $X=\{1,2,3,4\}$, and if $\preceq$ is defined on $X$ by $1\preceq 3$, $4\preceq 3$, $3\preceq 2$, and $1\incomp 4$ then $\preceq$ has the linear filter property for $\leq$, but does not have the CI-property for~$\leq$.
\end{rem}

\section{Order-preserving semilattice operations on finite chains}\label{sec:finite}
Let $n\geq 0$ be an integer, and denote by $(X_n, \leq_n)$ a totally ordered set of cardinality $n$. We assume that $X_n=\{1, \ldots, n\}$ and $\leq_n$ is the natural order $1<_n 2 <_n \cdots <_n n$. In particular, we have $X_0=\varnothing$ and $\leq_0=\varnothing$. If no confusion arises, we often write $\leq$ for $\leq_n$.  In this section, we provide a characterization (Theorem \ref{thm:main}) of   semilattice operations that are $\leq_n$-preserving.  This characterization is stated in terms of  properties of the Hasse diagram of $\preceq$. In Subsection \ref{sec:enum}, we also prove that the number of semilattice operations that are $\leq_n$-preserving  is the $n^{\text{th}}$-Catalan number, providing yet another realization of the sequence of Catalan numbers. We also provide an algorithm to test associativity for idempotent, symmetric, and $\leq_n$-preserving operations. Finally, given a binary semilattice  order $\preceq$ on $X_n$,  we consider  in Subsection \ref{sec:const} the problem of constructing the total orders on $X_n$ for which $\preceq$ is nondecreasing.

In what follows, by  \emph{binary tree} we mean an unordered rooted tree in which every vertex has at most two children. A \emph{binary forest} is a graph whose connected components are binary trees.

\subsection{Hasse diagram of finite nondecreasing semilattice orders}\label{sec:hasse}
\begin{lem}\label{lem:tree}
Let $(X_n, \preceq)$ be a  semilattice. The following conditions are equivalent.
\begin{enumerate}[(i)]
\item\label{it:mfd01} The order $\preceq$ has the linear filter property and there are no pairwise incomparable elements $a, b, c$ of $X$ such that $a\curlyvee b=b\curlyvee c$.
\item\label{it:mfd02} The Hasse diagram of $(X,\preceq)$ is a binary tree.
\end{enumerate}
\end{lem}
\begin{proof}
(\ref{it:mfd01}) $\implies$ (\ref{it:mfd02}) The Hasse diagram $G$ of $(X,\preceq)$ is connected since $\preceq$ is a semilattice order. The existence of a cycle in $H$ would imply the existence of two $\preceq$-incomparable elements with a lower bound, in contradiction with the linear filter property by Lemma \ref{lem:linf}. We have proved that $G$ is a tree, and we consider the top element of $(X,\preceq)$ as the root of $G$. Now,  if  $d$ is a vertex of $G$ with at least three children $a,b,c$, then $a, b, c$ are three $\preceq$-incomparable elements such that $d=a\curlyvee b=b\curlyvee c$, in contradiction with (\ref{it:mfd01}). We have proved that $G$ is a binary tree.

(\ref{it:mfd02}) $\implies$ (\ref{it:mfd01}) We prove the contrapositive. If $(X,\preceq)$ does not have the linear filter property, then by Lemma \ref{lem:linf} there are two incomparable elements $a,b$ that have a lower bound $c$. Then $c$, $a$, $b$ and $a\curlyvee b$ are vertices of a cycle in the Hasse diagram $G$ of $(X,\preceq)$.

Now, assume that $X$ has three incomparable elements $a, b, c$ such that $a\curlyvee b=b\curlyvee c$. Then, there are children $a', b', c'$ of $a\curlyvee b=b\curlyvee c$ such that $a\preceq a'$, $b\preceq b'$, and $c\preceq c'$, and $G$ is not a binary tree.
\end{proof}

If $\preceq$ is a semilattice order on $X_n$ that satisfies one of the conditions of Lemma \ref{lem:tree} (for instance, if $\preceq$ is nondecreasing  for $\leq_n$), then the Hasse diagram of $(X_n,\preceq)$ is a binary tree $G$. In what follows, we assume that $G$ is rooted by the top element of $(X_n, \preceq)$.

\begin{lem}\label{lem:str}
Let $(X_n, \preceq)$ be a semilattice. If $\preceq$ has the CI-property for $\leq_n$,
 then the following conditions are equivalent.
\begin{enumerate}[(i)]
\item\label{it:lvf01} The order $\preceq$ is internal for $\leq_n$.
\item\label{it:lvf02} If $x'$ is a child of $x$, then $x=\min\{z | z >y \text{ for all } y \in (x']_\preceq\}$  or $x=\max\{z | z <y \text{ for all } y \ \in (x']_\preceq\}$.
\item\label{it:lvf03} If $x_1$ and $x_2$ are two chlidren of a vertex $x$  in the Hasse diagram of $(X_n, \preceq)$, then there are $i\neq j$ in $\{1,2\}$ such that $x$ is an upper bound of $(x_i]_\preceq$ and a lower bound of $(x_j]_\preceq$ in $(X_n, \leq_n)$.

\end{enumerate}
\end{lem}
\begin{proof}
(\ref{it:lvf01}) $\implies$ (\ref{it:lvf02}) By Lemma \ref{lem:ic}, we have that $x$ is a lower bound or an upper bound of $(x']_\preceq$. Assume that $x$ is a lower bound of $(x']$ (the other case can be dealt with similarly). If $x$ has only one child, then by CI-property we have  $x=\max\{z | z <y \text{ for all } y  \in (x']_\preceq\}$. If $x$ has two children $x'$ and $x''$, then we obtain by internality that $y<x<z$ for every $y\preceq x''$ and $z\preceq x'$. If follows that $x=\max\{z | z <y \text{ for all } y \ \in (x']_\preceq\}$ and $x=\min\{z | z <y \text{ for all } y \ \in (x'']_\preceq\}$. 

(\ref{it:lvf02}) $\implies$ (\ref{it:lvf03}) Obvious.

(\ref{it:lvf03}) $\implies$ (\ref{it:lvf01}) We prove that $\preceq$ satisfies condition (\ref{it:mlo03}) of Lemma \ref{lem:int}. Let $x_1$ and $x_2$ be $\preceq$-incomparable elements, and assume that $x_1 <x_2$. Let $x'_1$ and $x'_2$ be the children of $x_1\curlyvee x_2$ such that $x_1\preceq x'_1$, and $x_2\preceq x'_2$. We obtain by (\ref{it:lvf02}) that $x_1\curlyvee x_2$ is an upper bound in $(X_n, \leq_n)$ of $(x'_1]_\preceq$ and a lower bound of $(x'_2]_\preceq$, which shows that $x_1<x_1 \curlyvee x_2<x_2$.
\end{proof}

The next result follows directly from Lemma \ref{lem:str}.

\begin{cor}\label{cor:main2}
 If $\preceq$ is a semilattice order that is nondecreasing for $\leq_n$, then its top element $r$ has only one child in the Hasse diagram of $(X_n,\preceq)$ if and only if $r\in\{1,n\}$.
\end{cor}

As stated in the next result, a similar equivalence as in Lemma \ref{lem:str} holds for semilattice orders that satisfy the linear filter property.

\begin{lem}\label{lem:str02}
Let $\preceq$ be a semilattice order on $X_n$ that has the linear filter property. Then, conditions (\ref{it:lvf01}) and (\ref{it:lvf03}) of Lemma \ref{lem:str} are equivalent.
\end{lem}
\begin{proof}
(\ref{it:lvf01}) $\implies$ (\ref{it:lvf03}) By internality, we know that $x$ lies between $x_1$ and $x_2$ in $(X_n,\leq_n)$. Assume that $x_1 < x < x_2$ (the case  $x_2 < x < x_1$ is obtained by symmetry). By the linear filter property we have $(x_1]_{\preceq}\cap (x_2]_{\preceq} = \varnothing$. Also, by the internality condition, there is no $y,z \in X_n$ such that $y,z<x$ (resp. $y,z>x$), $y \in (x_1]_\preceq$, and $z \in (x_2]_\preceq$. It follows that $x$ is an upper bound of $(x_1]_\preceq$ and a lower bound of $(x_2]_\preceq$.

(\ref{it:lvf03}) $\implies$ (\ref{it:lvf01}) The proof is the same as in Lemma \ref{lem:str}.
\end{proof}

\begin{lem}\label{lem:bez}
Let $(X_n, \leq_n)$ be a finite chain, and $\preceq$ be a semilattice order on $X_n$ with top element $r$.
\begin{enumerate}
\item\label{it:lio01} If $\preceq$ has the CI-property for $\leq_n$, and if $r$ has only one child in the Hasse diagram of $(X_n,\preceq)$, then $r\in\{1,n\}$.
\item\label{it:lio02} If $\preceq$ is internal for $\leq_n$, and if $r$  is either $1$ or $n$, then $r$ has only one child in the Hasse diagram of $(X_n,\preceq)$.
\end{enumerate}
\end{lem}
\begin{proof}
(\ref{it:lio01}) If $x$ is the child of $r$, then $(x]_\preceq$ is a convex subset of $(X_n,\leq_n)$ with $n-1$ elements.

(\ref{it:lio02}) We prove the contrapositive. Assume that $\preceq$ is internal for $\leq_n$ and that $x_1$ and $x_2$ are two children of $r$ in $(X_n,\preceq)$. By internality, we know that $x$ lies in between $x_1$ and $x_2$ in $(X_n,\leq)$, which shows that $x\not\in\{1,n\}$.
\end{proof}

The following result follows immediately from Lemmas \ref{lem:bez}, \ref{lem:str}, and \ref{lem:str02}.

\begin{cor}\label{cor:int}
Let $(X_n,\preceq)$ be a semilattice order  with top element $r$. Assume that $\preceq$ is internal  for $\leq_n$, and has the CI-property  for $\leq_n$ or the linear filter property. If $r$ has two children $x_1$, $x_2$ in the  Hasse diagram of $(X_n,\preceq)$, then $1$ and $n$ are $\preceq$-incomparable. Moreover, if $1\preceq x_1$ and $n\preceq x_2$, then $(x_1]_\preceq=\{1, 2, \ldots, r-1\}$ and $(x_2]_\preceq=\{r+1, r+2, \ldots, n\}$.
\end{cor}

\begin{thm}\label{thm:main}
Let $(X_n, \leq_n)$ be a finite totally ordered set and $\preceq$ be a semilattice order on $X_n$. The following conditions are equivalent.
\begin{enumerate}[(i)]
\item\label{it:nse01} The order $\preceq$ is nondecreasing for~$\leq_n$.
\item\label{it:nse02} $(X,\preceq)$ is a binary semilattice that satisfies  condition  (\ref{it:lvf02}) of Lemma \ref{lem:str}. 
\end{enumerate}
\end{thm}

\begin{proof}
(\ref{it:nse01}) $\implies$ (\ref{it:nse02}) follows from  Lemmas  \ref{lem:int},  \ref{lem:icintlin}, \ref{lem:tree}, and \ref{lem:str}.

(\ref{it:nse02}) $\implies$ (\ref{it:nse01}) By Lemma \ref{lem:ic} we obtain that $\preceq$ has the CI-property for $\leq_n$. It follows by Lemma \ref{lem:str} that $\preceq$ is internal for $\leq_n$.
%
%
\end{proof}

\begin{ex}\label{ex:main}
Let $\preceq$ be a semilattice order that is nondecreasing for $\leq_4$. According to Theorem \ref{thm:main}, its Hasse diagram is isomorphic to one of the binary trees depicted in  Fig.~\ref{fig:x4}, and $\preceq$ is one of the  orders defined by  the following labellings in  Fig.~\ref{fig:x4}: $(u,v,w,r)\in\{(1,3,2,4), (2,4,3,1)\}$, or \[(x,y,z,t)\in \{(3,4,1,2), (4,3,1,2),(1,2,4,3),(2,1,4,3)\},\] or
\begin{multline*}
(a,b,c,d)\in\{(1,2,3,4), (2,1,3,4), (2,3,1,4), (2,3,4,1),\\ (3,4,2,1), (3,2,1,4), (3,2,4,1), (4,3,2,1)\}.
\end{multline*}

\begin{figure}
\hfill
\begin{tikzpicture}[scale=\tkzscl]
\draw (0,0) node{$\bullet$} node[left]{$a$} -- (0,1) node{$\bullet$} node[left]{$b$} -- (0,2) node{$\bullet$} node[left]{$c$} -- (0,3) node{$\bullet$} node[left]{$d$} ;
\end{tikzpicture}
\hfill
\begin{tikzpicture}[scale=\tkzscl]
\draw (0,0) node{$\bullet$} node[above]{$t$} -- (-1,-1) node{$\bullet$} node[left]{$z$};
\draw (0,0) -- (1,-1) node{$\bullet$} node[right]{$y$} -- (1,-2) node{$\bullet$} node[right]{$x$} ;
\end{tikzpicture}
\hfill
\begin{tikzpicture}[scale=\tkzscl]
\draw (0,0) node{$\bullet$} node[above]{$r$} -- (0,-1) node{$\bullet$} node[left]{$w$}  -- (-1,-2) node{$\bullet$} node[left]{$u$};
\draw (0,-1) -- (1,-2) node{$\bullet$} node[right]{$v$} ;
\end{tikzpicture}
\hfill
\caption{Hasse diagrams of semilattices that are  nondecreasing for $\leq_4$.}\label{fig:x4}
\end{figure}
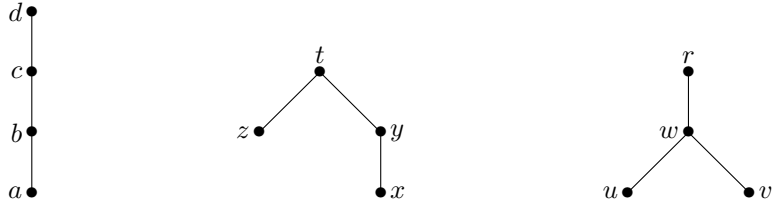
\end{ex}

Recall that  $e\in X_n$ is a \emph{neutral element} of an operation $F\colon X_n^2 \to X_n$ if $F(x,e)=F(e,x)=x$ for every $x\in X$. Observe that a finite semilattice $(X_n, \preceq)$ has a neutral element $e$ if and only if $e$ is a lower bound of $X_n$. The following result follows from the latter observation and Theorem \ref{thm:main}.

\begin{cor}
Let $\curlyvee\colon X_n^2 \to X_n$ be a $\leq_n$-preserving  semilattice operation. Then, $\curlyvee$ has a neutral element if and only if its associated order is a total order that is single-peaked for $\leq_n$.
\end{cor}

Theorem \ref{thm:main} enables us to give the isomorphism types of semilattices that are nondecreasing for $\leq_n$.
\begin{cor} \label{cor:type}
The isomorphism types of semilattices that are nondecreasing for $\leq_n$ and the isomorphism types of semilattices that have the linear filter property and are internal for $\leq_n$ coincide, and are the binary trees.
\end{cor}
\begin{proof}
It follows from Theorem  \ref{thm:main}  and Lemmas \ref{lem:int}  and  \ref{lem:tree} that any semilattice that  is nondecreasing for $\leq_n$, or that has the linear filter property and is internal for $\leq_n$ is a binary semilattice. Since any semilattice that is nondecreasing for $\leq_n$ is internal for $\leq_n$ and has the linear filter property, it suffices to show that if $G$ is a binary tree with $n$ vertices then there is a labeling of the vertices turning $G$ into the Hasse diagram of a semilattice that is nondecreasing for $\leq_n$. We proceed by induction on $n\geq 1$. For the induction step, if the root $r$ of $G$ has only one child, then we define a labeling of the vertices of $G$ by labeling $r$ with $n$, and labeling the vertices of $G-r$ with $1, \ldots, n-1$ using induction hypothesis. If $r$ has two children $x_1$ and $x_2$, let $C_i$ be the connected component of $G-r$ that contains $x_i$ for $i\in\{1,2\}$. We define a labeling of $G$ by labeling $r$ by $|C_1|+1$, and we label the vertices of $C_1$ and $C_2$ by $1, \ldots, |C_1|$ and $|C_1|+2, \ldots, n$, respectively, using induction hypothesis.
\end{proof}

 We apply Theorem \ref{thm:main}  to obtain the characterization of the  ``smooth''  semilattice operations given in \cite{Fodor2000}.

\begin{defn}[{{{{\cite{Fodor2000}}}}}]
An operation $F\colon X_n^2\to X_n$ is \emph{smooth} if $F(x,y)  \leq_n F(x+1,y) \leq_n F(x,y)+1$ for every $x, y \in X_n$ such that $x\neq n$, and  $F(x,y)  \leq_n F(x,y+1) \leq_n F(x,y)+1$ for every $x, y \in X_n$ such that $y\neq n$.
\end{defn}

Smooth semilattice operations can be characterized in terms of their associated order as follows.
\begin{cor}[{\cite{Fodor2000}}]
A semilattice operation $F\colon X_n^2 \to X_n$ is  smooth if and only if there exists $a\in X_n$ such that
\begin{gather*}
1 \prec_F 2 \prec_F \cdots \prec_F a-1\prec_F a,\\
n \prec_F n-1 \prec_F \cdots \prec_F a+1\prec_F a,
\end{gather*}
and $1\incomp n$.
\end{cor}
\begin{proof}
(Necessity) By Theorem \ref{thm:main}, it suffices to prove that if $x\prec y$ and there is no $z$ such that $x\prec z \prec y$ then $x\in\{y-1,y+1\}$. We prove the contrapositive. Assume that $x\not\in \{y-1,y+1\}$ and $x\prec y$. We can assume that $x=y-k$ for $k\geq 2$ (the case $x=y+k$ for $k\geq 2$ can be dealt with similarly). We have $F(y-k,y)=y$, so by smoothness and internality $F(y-k,y-1)=y-1$. Since $F$ is idempotent and  $\leq_n$-preserving, we also obtain $F(y-1,y)=y$. It follows that $x\prec y-1\prec y$.

(Sufficiency) Obvious.

\end{proof}

\subsection{Contour plots of idempotent operations}
Let $F\colon X_n^2 \to X_n$ be an idempotent operation.
The \emph{contour plot} of  $F$ is the simple graph $(X_n^2, E)$ where $E$ contains an edge between two distinct vertices $(x,y)$ and $(z,t)$ if and only $F(x,y)=F(z,t)$. In a drawing of a contour plot, we do not draw edges that can be obtained from existing ones by transitivity.
For every $z\in X$, we denote by $\deg_F(z)$ the number of elements in $F^{-1}(z)$.
It is convenient to define an idempotent operation by  providing a drawing of its contour plot. For instance, the operation $F\colon X_4^2 \to X_4$ defined by $F(1,2)=F(2,1)=F(1,1)=1$, $F(2,3)=F(3,2)=F(3,3)=3$, $F(1,3)=F(3,1)=F(2,2)=2$, and $F(x,4)=F(4,x)=4$ for $x=1,2,3,4$ is symmetric and idempotent. Its contour plot is depicted in Figure~\ref{fig:count}.

\setlength{\unitlength}{5.0ex}
\begin{figure}[htbp]
\begin{center}
\begin{small}
\begin{picture}(6,6)
\put(0.5,0.5){\vector(1,0){5}}\put(0.5,0.5){\vector(0,1){5}}
\multiput(1.5,0.45)(1,0){4}{\line(0,1){0.1}}%
\multiput(0.45,1.5)(0,1){4}{\line(1,0){0.1}}%
\put(1.5,0){\makebox(0,0){$1$}}\put(2.5,0){\makebox(0,0){$2$}}\put(3.5,0){\makebox(0,0){$3$}}\put(4.5,0){\makebox(0,0){$4$}}
\put(0,1.5){\makebox(0,0){$1$}}\put(0,2.5){\makebox(0,0){$2$}}\put(0,3.5){\makebox(0,0){$3$}}\put(0,4.5){\makebox(0,0){$4$}}
\multiput(1.5,1.5)(0,1){4}{\multiput(0,0)(1,0){4}{\circle*{0.2}}}
\drawline[1](2.5,1.5)(1.5,1.5)(1.5,2.5)\drawline[1](2.5,2.5)(1.5,3.5)(3.5,1.5)\drawline[1](3.5,2.5)(3.5,3.5)(2.5,3.5)\drawline[1](1.5,4.5)(4.5,4.5)(4.5,1.5)
\put(1.75,1.75){\makebox(0,0){$1$}}\put(2.75,2.75){\makebox(0,0){$2$}}\put(3.75,3.75){\makebox(0,0){$3$}}\put(4.75,4.75){\makebox(0,0){$4$}}
\end{picture}
\end{small}
\caption{A symmetric and idempotent operation on $X_4$ that is not associative (contour plot)}
\label{fig:count}
\end{center}
\end{figure}
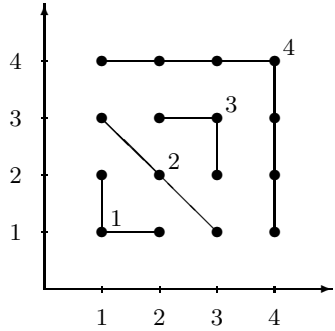


Recall that $a\in X$ is called a \emph{zero element} of an operation $F\colon X^2 \to X$ if $F(a,x)=F(x,a)$ for every $x\in X$.
\begin{rem}\begin{enumerate}\renewcommand{\itemsep}{0.5em}
\item Let $F\colon X_n^2 \to X_n$ be an idempotent and symmetric operation. If $F$ is associative, then the top element of $(X,\preceq_F)$ is a zero element of $F$. It is not difficult to see that the converse statement holds for every $n\leq 3$. However, it is not true in general. Consider for instance the operation $F\colon X_4 \to X_4$ defined in Fig.~\ref{fig:count}. The operation $F$ is idempotent and symmetric, and has 4 as the zero element. But it is not associative since $F|_{X_3^2}\colon X_3^2 \to X_3$ has no zero element.

\item Let $F\colon X_n^2 \to X_n$ be an idempotent operation. If $a\in X_n$ is a zero element of $F$, then $\deg_F(a)\geq 2n$. The converse statement does not hold. For instance, consider the idempotent and $\leq_n$-preserving operation $F\colon X_3^2 \to X_3$ whose contour plot is depicted in Figure \ref{fig:2}. We see that $\deg_F(2)=5$ but $F$ has no zero element.
\setlength{\unitlength}{5.0ex}
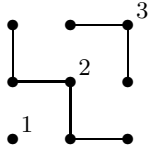
\begin{figure}[htbp]
\begin{center}
\begin{small}
\begin{picture}(3,3)
\multiput(0.5,0.5)(0,1){3}{\multiput(0,0)(1,0){3}{\circle*{0.2}}}
\drawline[1](0.5,2.5)(0.5,1.5)\drawline[1](0.5,1.5)(1.5,1.5)\drawline[1](1.5,1.5)(1.5,0.5)\drawline[1](1.5,0.5)(2.5,0.5)\drawline[1](1.5,2.5)(2.5,2.5)\drawline[1](2.5,1.5)(2.5,2.5)
\put(0.75,0.75){\makebox(0,0){$1$}}\put(1.75,1.75){\makebox(0,0){$2$}}\put(2.75,2.75){\makebox(0,0){$3$}}
\end{picture}
\end{small}
\caption{An idempotent and $\leq_3$-preserving operation on $X_3$}
\label{fig:2}
\end{center}
\end{figure}
\end{enumerate}

\end{rem}

The following Lemma shows how to recognize zero elements of $\leq_n$-preserving binary operations.
\begin{lem}\label{lem:dse}
Let $F\colon X_n^2 \to X_n$ be a $\leq_n$-preserving operation, and let $a\in X_n$.
\begin{enumerate}
\item\label{it:mpp01} The element $a$ is a zero element of $F$ if and only if $F|_{[1,a]\times [a,n]}$ and $F|_{[a,n]\times [1,a]}$ are the constant functions $(x,y)\mapsto a$.
\item \label{it:mpp01bis} If $F$ is idempotent, then $F(x,y)< a$ for every $(x,y) \in  [1,a-1]\times[1,a-1]$  and $F(x,y)>a$  for every $(x,y)\in [a+1,n]\times [a+1,n]$.
\item\label{it:mpp02} If $F$ is idempotent, then $a$ is a zero element of $F$ if and only if \begin{equation}\label{eqn:deg}
\deg_F(a)=2a(n-a+1)-1.
\end{equation}
\end{enumerate}
\end{lem}

\begin{proof}
(\ref{it:mpp01})
(Necessity) Let $(x,y) \in [1,a]\times [a,n]$ (the other case can be dealt with similarly). Since $F$ is $\leq_n$-preserving, we get
\[
a = F(x,a) \leq_n F(x,y) \leq_n F(a,y) = a,
\]
which shows that $F(x,y)=a$.

(Sufficiency) Obvious.

(\ref{it:mpp01bis})  Let $x,y\in [1,a-1]$ with $x<y$. Since $F$ is $\leq_n$-preserving and idempotent, we obtain
 $F(x,y)\leq F(y,y)=y<a$. The case $x,y\in [a+1,n]$ can be dealt with similarly.

(\ref{it:mpp02}) (Necessity) Let $a\in X_n$ be a zero element of $F$. Let us prove that
\begin{equation}\label{eqn:hty}
\{(x,y) \mid F(x,y)=a\}~=~\big([1,a]\times [a,n] \big) \cup \big([a,n]\times [1,a]\big)
\end{equation}
The ``$\supseteq$'' inclusion follows by (\ref{it:mpp01}). The reverse inclusion is the contrapositive of (\ref{it:mpp01bis}).
We derive \eqref{eqn:deg}  by a simple counting argument.

(Sufficiency) Let $F\colon X_n^2 \to X_n$ be a $\leq_n$-preserving idempotent operation, and let $a\in X_n$ be such that \eqref{eqn:deg} holds. By (\ref{it:mpp01bis}), we know that ``$\subseteq$'' inclusion of \eqref{eqn:hty} holds. From \eqref{eqn:deg}, we obtain that both sides of \eqref{eqn:hty} have the same (finite) cardinality,  which proves identity \eqref{eqn:hty}.
\end{proof}

We obtain the following corollary as an immediate consequence of Lemma \ref{lem:dse}.
\begin{cor}
Let $F\colon X_n^2 \to X_n$ be an idempotent and nondecreasing operation and let $a \in X_n$.
\begin{enumerate}
\item If $a\in \{1,n\}$, then $a$ is a zero element of $F$ if and only if $\deg_F(a)=2n-1$.
\item If $a$ is a zero element of $F$, then $a\in \{1,n\}$ if and only if $\deg_F(a)=2n-1$.
\end{enumerate}
\end{cor}
\begin{rem}
In \cite[Proposition 4]{Couceiro2018} it was shown that an element $e \in X_n$ is a neutral element of a  quasitrivial operation $F\colon X_n^2 \to X_n$ if and only if $\deg_{F}(e)=1$. We observe that if  we relax quastriviality into idempotency, then the sufficiency part of the previous equivalence does not hold. Indeed, considering the idempotent, symmetric, and $\leq_3$-preserving operation $F\colon X_3^2 \to X_3$ whose contour plot is depicted in Figure \ref{fig:3}, we see that $\deg_F(1)=\deg_{F}(3)=1$ but neither $1$ nor $3$ is a neutral element of $F$.
\end{rem}

\setlength{\unitlength}{5.0ex}
\begin{figure}[htbp]
\begin{center}
\begin{small}
\begin{picture}(3,3)
\multiput(0.5,0.5)(0,1){3}{\multiput(0,0)(1,0){3}{\circle*{0.2}}}
\drawline[1](0.5,2.5)(0.5,1.5)\drawline[1](0.5,1.5)(1.5,1.5)\drawline[1](1.5,1.5)(1.5,2.5)\drawline[1](1.5,2.5)(0.5,2.5)\drawline[1](1.5,1.5)(2.5,1.5)\drawline[1](2.5,1.5)(2.5,0.5)\drawline[1](2.5,0.5)(1.5,0.5)\drawline[1](1.5,0.5)(1.5,1.5)
\put(0.75,0.75){\makebox(0,0){$1$}}\put(1.75,1.75){\makebox(0,0){$2$}}\put(2.75,2.75){\makebox(0,0){$3$}}
\end{picture}
\end{small}
\caption{An idempotent, symmetric, and $\leq_3$-preserving operation on $X_3$}
\label{fig:3}
\end{center}
\end{figure}
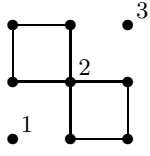

\subsection{Enumeration of finite $\leq_n$-preserving semilattice operations}\label{sec:enum}
It is shown in \cite[Theorem 3.3]{Couceiro2017}  that if  $F\colon X_n^2 \to X_n$ is an associative, symmetric, and quasitrivial operation, then $\preceq_F$ can be recovered from the degree sequence $(\deg_F(x))_{x\in X_n}$ of $F$ using the equivalence
\begin{equation}\label{eqn:eq}
x \preceq_F y \quad \iff \quad \deg_{F}(x)\leq \deg_F(y).
\end{equation}
Clearly, \eqref{eqn:eq} does not hold if $F$ is a $\leq_n$-preserving semilattice operation that is not quasitrivial, since \eqref{eqn:eq} implies that $\preceq_F$ is a total order. In this subsection, we provide a proper way to easily recover the order $\preceq_F$ associated with a $\leq_n$-preserving semilattice operation $F\colon X_n^2 \to X_n$ given by its contour plot. Moreover, we prove that the number of $\leq_n$-preserving semilattice operations  is  the $n^{\text{th}}$ Catalan number.

%
%
%

 Corollary \ref{cor:int}  and Theorem \ref{thm:main} enable us to construct the semilattice order $\preceq_F$ associated with a $\leq_n$-preserving semilattice operation $F\colon X^2 \to X_n$ with top element $r$ in a  recursive way, as follows:
\begin{enumerate}
\item\label{it:daz01} if $r\in \{1,n\}$ then $F'=F|_{(X_n\setminus\{r\})^2}$ is a $\leq_n$-preserving semilattice operation, and $\preceq_F$ is obtained by adding $r$ as the top element in $\preceq_{F'}$,

\item\label{it:daz02} if $r\not\in\{1,n\}$, then $F_1=F|_{[1, r-1]^2}$ and $F_2=F|_{[r+1, n]^2}$ are $\leq_n$-preserving semilattice operations, and $\preceq_F$ is obtained by adding $r$ as the top element of $\preceq_{F_1} \cup \preceq_{F_2}$.
\end{enumerate}

The preceding procedure can also be used  to test associativity for a $\leq_n$-preserving, symmetric, and idempotent operation $F\colon X_n^2 \to X_n$. If $F$ is not associative, then at some step $i$ of the procedure,  the function $F|_{X'^2}$, will satisfy one of the following conditions, where $X'$ is the set of elements that have not yet been added to $\preceq_F$ at step $i$:
\begin{enumerate}[(a)]
\item\label{it:rty} $F|_{X'^2}$ has no zero element,
\item $F|_{X'^2}$ has a  zero element and its contour plot $G$ has one or two  convex connected components $C_1$ and $C_2$ (see Lemma \ref{lem:dse})  but  there is $C_i\neq \varnothing$ such that $F|_{C_i}$ is not associative.
\end{enumerate}

Let $F\colon X_n^2 \to X_n$ be an idempotent, symmetric, and $\leq_n$-preserving operation. To test if $F$ is associative, apply the recursive procedure described by  (\ref{it:daz01}) and (\ref{it:daz02}). If at every step of the procedure such that $X'\neq \varnothing$  condition (\ref{it:rty}) is not satisfied,  then $F$ is associative and the procedure outputs $\preceq_F$. Otherwise, the procedure stops at a some step where (\ref{it:rty}) is satisfied, and $F$ is not associative.

For every $n\geq 0$, let $ \alpha(n)$ be the number of join-semilattice orders  on $X_n$ that are nondecreasing for $\leq$. By definition, we have $ \alpha(0)=1$.  The following Proposition proves that $ \alpha(n)$ is the $n^{th}$ Catalan number (see, \emph{e.g.}, \cite{Stanley1999}). We denote by $\N=\{0,1,2\ldots\}$ the set of nonnegative integers.

\begin{prop}\label{prop:cat}
The sequence $ \alpha\colon \N \to \N$ satisfies the recurrence relation
\begin{equation}\label{eqn:cat}
 \alpha(n)=\sum_{i=1}^n  \alpha(n-i) \alpha(i-1), \qquad n\geq 1.
\end{equation}
It follows that $ \alpha(n)$ is the $n^{th}$ Catalan number $\frac{(2n)!}{n!(n+1)!}$ for every $n\in \N$.
\end{prop}
\begin{proof}
Let $\preceq$ be a semilattice order on $X_n$ that is nondecreasing for $\leq_n$. By Theorem \ref{thm:main}, we know that $(X_n, \preceq)$ is a binary semilattice. Let $r$ be the top element of its Hasse diagram, and  set $X'=X_n\setminus\{r\}$. By Corollary \ref{cor:main2}, if $r\in \{1,n\}$, then $\preceq_{X'}$ is one of the $u(n-1)$ semilattice orders that are nondecreasing for $\leq_{X'}$. By Corollaries \ref{cor:main2} and \ref{cor:int}, if $r\not\in\{1,n\}$, then $\preceq_{X'}$ is the  union of one of the $ \alpha(r-1)$ semilattices orders on $[1,r-1]$ that is nondecreasing for  $\leq_{[1,r-1]}$ with one of the $ \alpha(n-r)$ semilattice orders on $[r+1,  n]$ that is nondecreasing for $\leq_{[r+1,  n]}$.
\end{proof}
Proposition \ref{prop:cat} counts the number of semilattice orders on $X_n$ that have the CI-property and are internal for $\leq_n$. The Hasse diagram of these semilattices are binary trees verifying condition (\ref{it:nse02}) of Theorem \ref{thm:main}.  In Appendix \ref{sec:ap}, we consider the problem of counting the semilattice orders whose Hasse diagram is a binary tree, and either are internal or have the CI-property for $\leq_n$.

\subsection{Construction of the total orders for that a  semilattice order is nondecreasing}\label{sec:const}

Let $\preceq$ be a semilattice order with top element $r$. The family of total orders $\leq$ on $X_n$ for which $\preceq$ is nondecreasing can be constructed by {recursion}, using the following result.

\begin{prop}\label{prop:hty}
Let $(X_n,\preceq)$ be a binary semilattice  with top element $r$ and $G$ be its Hasse diagram. Let $C_1$ and $C_2$ be the connected components of $G-r$,  with the convention that $C_2=\varnothing$ if $r$ has only one child. The following conditions are equivalent.
\begin{enumerate}[(i)]
\item\label{it:loi01} The order $\preceq$ is nondecreasing for $\leq_n$.
\item\label{it:loi02} There exist total orders $\leq_1$ on $C_1$ and $\leq_2$ on $C_2$ such that
\begin{enumerate}[(a)]
\item the order $\preceq_{C_i}$ is nondecreasing for $\leq_i$ for every $1 \leq i\leq 2$,
\item the total order $\leq$ is obtained  by adding $r$ as the top of $\leq_{1}$ and the bottom of $\leq_{2}$, or conversely.
\end{enumerate}
\end{enumerate}
\end{prop}

\begin{proof}
(\ref{it:loi02}) $\implies$ (\ref{it:loi01}) Since $\preceq_{C_1}$ and $\preceq_{C_2}$ have the CI-poperty for $\leq_1$ and $\leq_2$, respectively, it follows by Lemma  \ref{lem:ic} that $\preceq$ has the CI-property with respect to $\leq$. Similarly, we obtain by Lemma \ref{lem:int} (\ref{it:mlo02}) that $\preceq$ is internal for $\leq$.

(\ref{it:loi01}) $\implies$ (\ref{it:loi02})
The proof is obtained by an easy induction on $n$, using Lemmas \ref{lem:bez} and \ref{cor:int} in the induction step.
\end{proof}

The following corollary is obtained from Proposition \ref{prop:hty} by an easy induction on $n$.

\begin{cor}\label{cor:rev}
Let $(X_n,\preceq)$ be a binary semilattice, and let $L$ be the number of minimal elements in $(X_n,\preceq)$. The number of total orders for which  $\preceq$ is nondecreasing is equal to $2^{n-L}$.
\end{cor}

\begin{ex}
The eight total orders on $X=\{a,b,c,d,r\}$ for which the semilattice order $\preceq$ depicted in Fig.~\ref{fig:bin} is nondecreasing are
\begin{gather*}
r<b<a<c<d, \qquad r<b<a<d<c,\\
r<c<d<a<b, \qquad r<d<c<a<b,
\end{gather*}
and their dual orders.
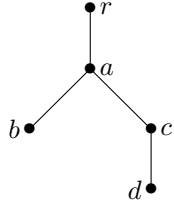
\begin{figure}
\begin{tikzpicture}[scale=\tkzscl]
\draw (0,1) node{$\bullet$} node[right]{$r$} -- (0,0);
\draw (0,0) node{$\bullet$} node[right]{$a$} -- (-1,-1) node{$\bullet$} node[left]{$b$};
\draw (0,0) -- (1,-1) node{$\bullet$}node[right]{$c$} -- (1,-2)  node{$\bullet$} node[left]{$d$};

\end{tikzpicture}
\caption{semilattice $(X, \preceq)$ whose Hasse diagram} is a binary tree\label{fig:bin}
\end{figure}
\end{ex}

It follows from Corollary \ref{cor:type} that the number  $\tau(n)$ of isomorphism types of  semilattices that are $\leq_n$-preserving is equal to the number $A001190(n+1)$  of unordered and unlabeled rooted binary trees (see \cite{OEIS}), where a tree is said to be \emph{unordered} if no order is specified on the children of a parent vertex). A similar counting argument  as in Proposition \ref{prop:cat}  proves the following recurrence relations for $\tau$.
\begin{cor}\label{cor:tnum}
The number $\tau(n)$ of isomorphism types of semilattices that are nondecreasing for $\leq_n$ satisfies $\tau(0)=1$, $\tau(1)=1$ and
\begin{gather*}
\tau(2n)=\sum_{i=0}^{n-1} \tau(i)\tau(2n-1-i)\\
\tau(2n+1)=\sum_{i=0}^{n-1}\tau(i)\tau(2n-i)+\frac{\tau(n)}{2}\left(\tau(n)+1\right)
\end{gather*}
for all $n\geq 1$.
\end{cor}

\section{$k$-ary symmetric, idempotent and $\leq_n$-preserving semigroups}\label{sec:kary}
Let $k\geq 2$ be an integer. Recall that an operation $F\colon X^k \to X$ is \emph{idempotent} if it satisfies the equation $F(x,\ldots, x)=x$. It is \emph{symmetric} if it satisfies the equations $F(x_1, \ldots, x_k)=F(x_{\sigma1},  \ldots, x_ {\sigma k})$ for every bijection $\sigma$ of $\{1,  \ldots, k\}$. If $X$ is equipped with a total order $\leq$, then $F$ is said to be \emph{$\leq$-preserving} if $F(x_1, \ldots, x_k)\leq F(x'_1, \ldots, x'_k)$ for every $x_1, x'_1, \ldots, x_k,x'_k$ such that $x_i\leq x'_i$ for every $i \leq k$.

In this  section, we characterize  associative, idempotent, symmetric, and $\leq$-preserving operations $F\colon X^k \to X$ on a chain $(X,\leq)$, where associativity for $k$-ary operations is defined as follows.

\begin{defn}
An operation $F\colon X^k\to X$ is said to be
 \emph{associative} if it satisfies the equation
\begin{eqnarray*}
\lefteqn{F(x_1,\ldots,x_{i-1},F(x_i,\ldots,x_{i+k-1}),x_{i+k},\ldots,x_{2k-1})}\\
&=& F(x_1,\ldots,x_i,F(x_{i+1},\ldots,x_{i+k}),x_{i+k+1},\ldots,x_{2k-1})
\end{eqnarray*}
for every $i\leq k$.

\end{defn}
If $F\colon X^k \to X$ is an associative operation, then $(X,F)$ is called a \emph{$k$-ary semigroup} \cite{Dor28,Pos40}. Examples of $k$-ary associative operations are given by compositions of binary associative ones. More precisely, if $H\colon X^2 \to X$ is  associative, and if $(H_k)_{k\geq 2}$ is the sequence of the $k$-ary operations $H_k\colon X^k \to X$ defined inductively by the rules
\begin{gather}
H_2(x,y)=H(x,y),\nonumber\\
 H_{k+1}(x_1, \ldots, x_{k+1})=H(H_{k}(x_1, \ldots, x_{k}), x_{k+1}), \qquad k\geq 3,\label{eqn:dec}
\end{gather}
then $H_k$ is a $k$-ary associative operation for every odd $k\geq 2$.

\begin{defn}[{{\cite{Ackerman,DudMuk06}}}]
An associative operation $F\colon X^k \to X$ that can be obtained from an associative operation $H\colon X^2 \to X$ as in \eqref{eqn:dec} is said to be \emph{reducible to} (or to \emph{be derived from}) $H$.
\end{defn}

Examples of $k$-ary associative operations $F\colon \R^k \to \R$ that are not reducible to a binary associative one are given by the alternating sums $(x_1, \ldots, x_k) \mapsto \sum_{i=1}^k (-1)^i x_k$, for every odd $k \geq 3$.

\begin{thm}[{\cite[Theorem 4.3]{Kiss}}]\label{prop:red}
Let $(X,\leq)$ be a totally ordered set, and let $F\colon X^k \to X$ be an associative, idempotent, symmetric, and $\leq$-preserving operation. Then $F$ is reducible to a unique binary operation $G\colon X^2 \to X$ that is defined by $G(x,y)=F(x, \ldots,  x,y)$.
\end{thm}

By combining Theorem \ref{prop:red} and Theorem \ref{cor:main} we obtain the following characterization.

\begin{thm}\label{thm:kary}
Let $(X,\leq)$ be totally ordered set. An operation $F\colon X^k \to X$ is associative, idempotent, symmetric, and $\leq$-preserving if and only if $F$ is reducible to the  join operation of a semilattice order that is nondecreasing for $\leq$.
\end{thm}

Combining  Proposition \ref{prop:cat} and Theorem \ref{thm:kary}, we obtain the following corollary.

\begin{cor}
The number of $k$-ary operations on $X_n$  that are associative, idempotent, symmetric, and $\leq_n$-preserving operation is the $n^{th}$ Catalan number.
\end{cor}
\section{Conclusions and further research}
In this paper, we have characterized $\leq$-preserving semilattice operations in terms of a property generalizing the single-peakedness property (Theorem \ref{thm:mainic}). We have also characterized finite $\leq$-preserving semilattices in terms of properties of their Hasse diagram (Theorem \ref{thm:main}).  In particular, we have proved that the number of nondecreasing semilattice operations on a $n$-element chain is the $n^{\text{th}}$ Catalan number (Proposition \ref{prop:cat}). Thus, the main results of this paper are new contributions to the problem of characterizing subclasses of associative operations.

We list below several open questions and topics of current research.
\begin{enumerate}[(I)]

\item Let $\preceq$ be a semilattice order on $X_n$. CI-property for $\leq_n$, internality for $\leq_n$, linear filter property, and the property of having a Hasse diagram that is a binary tree have straightforward first order translations in the language of semigroups. Find alternative characterizations of classes of idempotent and symmetric  semigroups whose semilattice orders satisfy some of these properties.

\item Find a closed-form expression for the number of semilattice orders on $X_n$ that are internal for $\leq_n$ and have the linear filter property (see Proposition \ref{prop:catbis}).

\item Find a closed-form expression for the number of semilattice orders on $X_n$ that have the CI-property for $\leq_n$, and whose Hasse diagram is a binary tree (see Proposition  \ref{prop:catter}).

\item Find a recurrence relation and a closed-form expression for the   number of semilattice orders on $X_n$ that are internal for $\leq_n$. The first elements of this sequence are $1, 1, 2, 7, 36, 247$.

\item Identify explicit one-to-one correspondences between the class of nondecreasing semilattice orders for $\leq_n$ and other classes  of objects that are counted by the Catalan numbers \cite{Stanley1999,Stanley2015}.

\item Find  characterizations corresponding to Theorems \ref{thm:mainic} and \ref{thm:main}  for semilattice operations that are $\leq$-preserving  on a poset $(X,\leq)$.

\end{enumerate}

\appendix
\section{Additional enumeration problems}\label{sec:ap}
In this Section we give recurrence relations on $n$ for the number $\beta(n)$ of semilattice orders on $X_n$ that are internal for $\leq_n$ and have the linear filter property, and the number $\delta(n)$ of binary semilattice orders that have the CI-property  for $\leq_n$.  We also give a recurrence relation on $\preceq$ for the number  $\gamma(\preceq)$ of total orders for which a binary semilattice order $\preceq$ is internal, and the number $\eta(\preceq)$ of total orders for which $\preceq$ has the CI-property.

By definition, we have $\beta(0)=\delta(0)=1$, $\beta(1)=\delta(1)=1$,  $\gamma(\varnothing)=\eta(\varnothing)=1$, and $\gamma(\preceq)=\eta(\preceq)=1$ if $\preceq$ is the only semilattice order on $X_1$. The following result provides a recurrence relation for $\beta(n)$. It turns out that $\beta(n) = A006014(n)$ for $n\geq 1$ (see \cite{OEIS}).

\begin{prop}\label{prop:catbis}
The sequence $\beta\colon \mathbb{N} \to \mathbb{N}$ satisfies the recurrence relation
\[
\beta(n)=\sum_{i=1}^{n-2} \beta(i) \, \beta(n-i-1) + n\, \beta(n-1), \qquad n\geq 2.
\]
\end{prop}
\begin{proof}
Let $\preceq$ be a semilattice order on $X_n$ that is internal for $\leq_n$ and has the linear filter property, let $r$ be the top element of $(X_n,\preceq)$, and  set $X'=X_n\setminus\{r\}$. If $r$ has only one child, then $\preceq_{X'}$ is one of the $\beta(n-1)$ semilattice orders that have the linear filter property and are internal for $\leq_{X'}$. If $r$ has two children, then $r\not\in\{1,n\}$ by Lemma  \ref{lem:bez} (\ref{it:lio02}). By Corollary \ref{cor:int}, the order  $\preceq_{X'}$ is the union of one of the $\beta(r-1)$ semilattices orders on $[1,r-1]$ that have the linear filter property and are internal for $\leq_{[1,r-1]}$ with one of the $\beta(n-r)$ semilattice orders on $[r+1,n]$ that have the linear filter property and are internal for $\leq_{[r+1,n]}$. Thus, we obtain
\begin{equation}\label{eqn:cat02}
\beta(n)=\sum_{i=1}^{n} \beta(i-1)\, \beta(n-i) + (n-2)\, \beta(n-1), \qquad n\geq 2,
\end{equation}
which concludes the proof.
\end{proof}
Observe that the recurrence relation \eqref{eqn:cat02} differs from the relation \eqref{eqn:cat} defining the Catalan numbers by the term $(n-2) \,\beta(n-1)$ only.

\begin{rem}
The number of isomorphism types of semilattice orders on $X_n$ that are internal for $\leq_n$ and have the linear filter property is equal to $\tau(n)$ (see Corollary \ref{cor:type}).
\end{rem}

A similar proof as for Proposition \ref{prop:hty} gives the following result.

\begin{prop}\label{prop:gde}
Let $(X_n,\preceq)$ be a binary semilattice ($n\geq 2$) with top element $r$. Let $C_1$ and $C_2$ be the connected components of $G-r$,  with the convention that $C_2=\varnothing$ if $r$ has only one child. The following conditions are equivalent.
\begin{enumerate}[(i)]
\item\label{it:loig01} The order $\preceq$ is internal for $\leq_n$.
\item\label{it:loig02} There exist total orders $\leq_1$ on $C_1$ and $\leq_2$ on $C_2$ such that
\begin{enumerate}[(a)]
\item the order $\preceq_{C_i}$ is internal for $\leq_i$ for every $1 \leq i\leq 2$,
\item if $C_2\neq \varnothing	$ then  $\leq$ is obtained  by adding $r$ as the top of $\leq_{1}$ and the bottom of $\leq_{2}$, or conversely, and if $C_2=\varnothing$ then $\leq$ is obtained by inserting $r$ anywhere in $\leq_{1}$.
\end{enumerate}
\end{enumerate}
\end{prop}

The following recurrence relation  can be obtained by induction on $\preceq$ using Proposition \ref{prop:gde}.

\begin{cor}\label{cor:de}
Let $(X_n,\preceq)$ be a binary semilattice ($n\geq 2$) with top element $r$.  Let $C_1$ and $C_2$ be the connected components of $G-r$,  with the convention that $C_2=\varnothing$ if $r$ has only one child. Then,
 \[
\gamma(\preceq)=2^{i-1}\, n^{2-i}\, \gamma(\preceq_{C_1})\,  \gamma(\preceq_{C_2}),
\]
where $i$ is the number of children of $r$.
\end{cor}

To provide a recurrence relation for the number $\delta(n)$ of binary semilattice orders that have the CI-property for $\leq_n$, we use the following lemma.
\begin{lem}\label{lem:IC}
Let $(X_n,\preceq)$ be a binary semilattice with top element $r$ such that $\preceq$ has the CI-property for $\leq_n$. If $r\in [2,n-1]$, then $r$ has two children $x_1$, $x_2$ in the Hasse diagram of $(X_n,\preceq)$, and $1$ and $n$ are $\preceq$-incomparable. Moreover, if $1\preceq x_1$ and $n\preceq x_2$, then $(x_1]_\preceq=\{1, 2, \ldots, r-1\}$ and $(x_2]_\preceq=\{r+1, r+2, \ldots, n\}$.
\end{lem}
\begin{proof}
We know by Lemma \ref{lem:bez} (\ref{it:lio01}) that $r$ has two children $x_1$ and $x_2$. By the CI-property, there is no $y,z \in X_n$ such that $y<r<z$ (resp. $z<r<y$) and $y,z \in (x_i]_\preceq$ for some $i\in \{1,2\}$. 
\end{proof}

\begin{prop}\label{prop:catter}
The sequence $\delta\colon \mathbb{N} \to \mathbb{N}$ satisfies the recurrence relation
\begin{equation}\label{eqn:ICB}
\delta(n)=\sum_{i=1}^{n} \delta(i-1)\,  \delta(n-i) + \sum_{j=1}^{n-2}{n-1 \choose j}\, \delta(j)\, \delta(n-j-1),
\end{equation}
for every $n \geq 1$.

\end{prop}
\begin{proof}
We say that $\preceq$ has ICB-property for $\leq$ if $\preceq$ is a binary semilattice order that has the CI-property for $\leq$.
Let $\preceq$ be a semilattice order on $X_n$ that has the ICB-property for $\leq_n$, let $r$ be the top element of $(X_n,\preceq)$, and  set $X'=X_n\setminus\{r\}$. By Lemma \ref{lem:bez}, if $r$ has only one child, then $r\in \{1,n\}$ and $\preceq_{X'}$ is one of the $\delta(n-1)$, semilattice orders that have the ICB-property for $\leq_{X'}$. Also, by Lemma \ref{lem:bez}, if $r\not\in\{1,n\}$ then $r$ has two children and by Lemma \ref{lem:IC} the relation $\preceq_{X'}$ is the union of one of the $\delta(r-1)$ semilattice orders on $[1,r-1]$ that have the ICB-property for $\leq_{[1,r-1]}$, with one of the $\delta(n-r)$ semilattice orders on $[r+1,n]$ that have the ICB-property for $\leq_{[r+1,n]}$. If $r\in \{1,n\}$ and $r$ has two children, then $\preceq_{X'}$ is the union of one of the $\delta(|(x_1]|)$ semilattice orders on $(x_1]$ that have the ICB-property for $\leq_{(x_1]}$ with one of the  $\delta(n-1-|[x_1)|)$ semilattice orders on $[x_2)$ that have the ICB-property for $\leq_{(x_2]}$.
\end{proof}
By simple algebraic manipulations, we obtain that \eqref{eqn:ICB} is equivalent to
\[
\delta(n+1)+2\, \delta(n)=\sum_{i=0}^n \left(\binom{n}{i}+1\right)\, \delta(i) \, \delta(n-i),
\]
for every $n\geq 2$.

\begin{table}
\[
\begin{array}{|c|c|c|c|c|}\hline
n &  \alpha(n) & \tau(n) & \beta(n) & \delta(n)\\ \hline
0 & 1 &  1 & 1 & 1\\
1 & 1 & 1 & 1 & 1\\
2 & 2 & 1 & 2 & 2\\
3 & 5 & 2 & 7 & 7\\
4 & 14 & 3 & 32 & 30\\
5 & 42 & 5 & 178 & 158\\
6 & 132 & 10 & 1160 & 984\\
7 & 429 & 21 & 8653 & 7129\\
8 & 1430 &42 & 72704 & 59026\\\hline
OEIS & A000108 & A001190 & A006014 & \\\hline
\end{array}
\]
\caption{Table of sequences $ \alpha$, $\tau$, $\beta$ and $\delta$, with their OEIS entries \cite{OEIS}}
\end{table}

Using Lemma \ref{lem:ic}, a similar proof as for Proposition \ref{prop:hty} gives the following result. Recall that if $(X,\leq)$ and $(X',\leq')$ are two disjoint  partially ordered sets, then the \emph{ordinal sum} $\leq\oplus\leq'$ is the order $\preceq$ defined on $X\cup X'$ by setting $y\preceq z$ if $y\leq z$, or $y\leq' z$, or $y\in X$ and $z\in X'$.

\begin{prop}\label{prop:gde2}
Let $(X_n,\preceq)$ be a binary semilattice order ($n\geq 2$) with top element $r$. Let $C_1$ and $C_2$ be the connected components of $G-r$,  with the convention that $C_2=\varnothing$ if $r$ has only one child. The following conditions are equivalent.
\begin{enumerate}[(i)]
\item\label{it:loi001} The order $\preceq$ has the CI-property for $\leq_n$.
\item\label{it:loi002} There exist total orders $\leq_1$ on $C_1$ and $\leq_2$ on $C_2$ such that
\begin{enumerate}[(a)]
\item the order $\preceq_{C_i}$ has the CI-property for $\leq_i$ for every $1 \leq i\leq 2$,
\item if $C_2\neq \varnothing$ then $\leq$ is obtained either by adding $r$ as the top of $\leq_{1}$ and the bottom of $\leq_{2}$, or conversely, or by adding $r$ as the top or the bottom of either $\leq_{1} \oplus \leq_{2}$, or $\leq_{1} \oplus \leq_{2}$, and if $C_2=\varnothing$ then $\leq$ is obtained by adding $r$ as the top or the bottom of $\leq_{1}$.
\end{enumerate}
\end{enumerate}
\end{prop}

The following recurrence relation can be proved by induction on $\preceq$ using Proposition \ref{prop:gde2}.

\begin{cor}
Let $(X_n,\preceq)$ be a binary semilattice ($n\geq 2$) with top element $r$ and let $C_1$ and $C_2$ be the connected components of $G-r$, with the convention that $C_2=\varnothing$ if $r$ has only one child. Then,
 \[
\eta(\preceq)=3^{i-1}\, 2\, \eta(\preceq_{C_1})\,  \eta(\preceq_{C_2}),
\]
where $i$ is the number of children of $r$.
\end{cor}

\end{document}